\documentclass[reqno,psamsfonts,10pt]{aarticle}
\usepackage[sort&compress]{natbib}
\usepackage{amsthm,amsmath,amssymb,booktabs,fancyhdr,layout,enumitem}
\usepackage[colorlinks,citecolor=DarkBlue,linkcolor=DarkBlue,anchorcolor=DarkBlue,urlcolor=DodgerBlue]{hyperref}
\usepackage{mathtools,mathscinet}
\usepackage{qpalatin}
\usepackage{tikz,graphicx}
\usetikzlibrary{matrix,arrows,calc}

\tikzset{every scope/.style={>=angle 60,thick}}
 \def\galoisskip{.8mm}
 \def\galoislength{7mm} 	
 \def\mprrskip{1mm}     	
  
 \def\longgaloislength{12mm} 
 \def\arrowlength{6mm}
 \def\arrowhight{2pt} 		
 
 \def\shortarrowlength{4.75mm}
 \def\longarrowlength{10mm}
 \def\horizontalskip{3pt}	
 \def\nodeposition{.5}		

\def\wmpr#1{\smash{\mathop{\hbox to 12pt{\rightarrowfill}}\limits^{#1}}}

\def\mpr#1{
\hspace{\horizontalskip}\begin{tikzpicture}[baseline=-1pt]
	\draw[->] (0,\arrowhight) -- node[above,pos=\nodeposition]{${\scriptstyle #1}$} (\arrowlength,\arrowhight);
\end{tikzpicture}
\hspace{\horizontalskip}}

\def\ir{
\hspace{\horizontalskip}\begin{tikzpicture}[baseline=-1pt]
	\draw[->] (0,\arrowhight) -- (\arrowlength,\arrowhight);
\end{tikzpicture}
\hspace{\horizontalskip}}

\def\sir{
\hspace{\horizontalskip}\begin{tikzpicture}[baseline=-1pt]
    \draw[->] (0,\arrowhight) -- (\shortarrowlength,\arrowhight);
\end{tikzpicture} \hspace{\horizontalskip}}

\def\lmpr#1{
\hspace{\horizontalskip}\begin{tikzpicture}[baseline=-1pt]
    \draw[->] (0,\arrowhight) -- node[above,pos=\nodeposition]{${\scriptstyle #1}$} (\longarrowlength,\arrowhight);
\end{tikzpicture} \hspace{\horizontalskip}}

\def\smrp#1{\hspace{\horizontalskip}\smash{\mathop{\hbox to 12pt{\rightarrowfill}}\limits_{#1}}\hspace{\horizontalskip}} 

\def\smrp#1{\hspace{\horizontalskip}\smash{\mathop{\hbox to 12pt{\rightarrowfill}}\limits_{#1}}\hspace{\horizontalskip}}

\def\lmprr#1#2{\begin{tikzpicture}[baseline=-4pt]
\hspace{\horizontalskip}
\begin{scope}
\draw[->] (0,\mprrskip) -- node[above,pos=\nodeposition]{${\scriptstyle #1}$} (\longgaloislength,\mprrskip);
\draw[->] (0,-\mprrskip) -- node[below,pos=\nodeposition]{${\scriptstyle #2}$} (\longgaloislength,-\mprrskip);
\end{scope}
\end{tikzpicture}
\hspace{\horizontalskip}}

\def\elmap#1{
\hspace{\horizontalskip}\begin{tikzpicture}[baseline=-1pt]
\draw[|->] (0,\arrowhight) -- node[above,pos=\nodeposition]{${\scriptstyle #1}$} (\arrowlength,\arrowhight);
\end{tikzpicture}
\hspace{\horizontalskip}}

\def\selmap#1{
\hspace{\horizontalskip}\begin{tikzpicture}[baseline=-1pt]
\draw[|->] (0,\arrowhight) -- node[above,pos=\nodeposition]{${\scriptstyle #1}$} (\shortarrowlength,\arrowhight);
\end{tikzpicture}
\hspace{\horizontalskip}}

\def\ellmap#1{
\hspace{\horizontalskip}
\begin{tikzpicture}[baseline=-1pt]
\draw[<-|] (0,\arrowhight) -- node[above,pos=\nodeposition]{${\scriptstyle #1}$} (\arrowlength,\arrowhight);
\end{tikzpicture}
\hspace{\horizontalskip}}

\def\elmapt#1{
\hspace{\horizontalskip}
\begin{tikzpicture}[baseline=-1pt]
\draw[|->] (0,\arrowhight) -- node[above,pos=\nodeposition]{${\scriptstyle #1}$} (\arrowlength,\arrowhight);
\end{tikzpicture}
\hspace{\horizontalskip}}

\def\galois#1#2{
\hspace{\horizontalskip}
\begin{tikzpicture}[baseline=-4pt]
\begin{scope}
\draw[->] (0,\galoisskip) -- node[above,pos=.5]{${\scriptstyle #1}$} (\galoislength,\galoisskip);
\draw[<-] (0,-\galoisskip) -- node[below,pos=.6]{${\scriptstyle #2}$} (\galoislength,-\galoisskip);
\end{scope}
\end{tikzpicture}
\hspace{\horizontalskip}}

\def\lgalois#1#2{
\hspace{\horizontalskip}
\begin{tikzpicture}[baseline=-4pt]
\begin{scope}
\draw[->] (0,\galoisskip) -- node[above,pos=.5]{${\scriptstyle #1}$} (\longgaloislength,\galoisskip);
\draw[<-] (0,-\galoisskip) -- node[below,pos=.5]{${\scriptstyle #2}$} (\longgaloislength,-\galoisskip);
\end{scope}
\end{tikzpicture}
\hspace{\horizontalskip}}

\definecolor{DarkBlue} 		{rgb} 	{0, 0, 0.545}
\definecolor{DodgerBlue} 	{rgb} 	{0.117, 0.564, 1}
\def\arrowhight{2.8pt}
\newcommand\Figref[1]{Figure \ref{#1}\ifthenelse{\value{page}=\pageref{#1}}{}{ on page \pageref{#1}}}
\newcommand\figref[1]{figure \ref{#1}\ifthenelse{\value{page}=\pageref{#1}}{}{ on page \pageref{#1}}}

\def\bt{\begin{center}\begin{tikzpicture}\matrix[matrix of math nodes,column sep=1cm, row sep=1cm]}
\def\bth{\begin{center}\begin{tikzpicture}\matrix[matrix of math nodes,column sep=.5cm, row sep=1cm]}
\def\et{\end{tikzpicture}\end{center}}
\def\btm#1#2{\begin{center}\begin{tikzpicture}\matrix[matrix of math nodes,column sep=#1, row sep=#2}
\def\bsc{\begin{scope}}
\def\esc{\end{scope}}

\def\|#1{\,\begin{tikzpicture}[baseline]\draw[-,thin](0,-3pt)--(0,6pt);\node[anchor=west] at (-1pt,-2pt) {$_{#1}\!\!\!$}; \end{tikzpicture}\,}

\def\Spec{\mathsf{Spec}}

\def\Hom{\mathrm{Hom}}

\newcommand{\id}{\operatorname{id}}
\newcommand{\can}{\operatorname{\it can}}

\def\bZ{\mathbb{Z}}

\def\bE{\mathbb{E}}
\def\bF{\mathbb{F}}
\def\bM{\mathbb{M}}

\def\Gal{\mathsf{Gal}}
\def\Fix{\mathsf{Fix}}

\def\Sub{\mathsf{Sub}}
\def\Id{\mathsf{Id}}
\def\coId{\mathsf{coId}}
\def\sup{\mathit{sup}}
\def\inf{\mathit{inf}}
\def\Quot{\mathsf{Quot}}
\def\qquot{\mathsf{Quot_{\textit{gen}}}}
\def\qsub{\mathsf{Sub_{\textit{gen}}}}
\def\id{\mathit{id}}
\def\can{\mathit{can}}

\def\Hom{\mathsf{Hom}}

\newtheorem{theorem}{Theorem}[section] 
\newtheorem*{theorem*}{Theorem}

\newtheorem*{mtheorem*}{Main Theorem}
\newtheorem{proposition}[theorem]{Proposition}
\newtheorem*{proposition*}{Proposition}

\newtheorem*{construction*}{Construction}
\newtheorem{definition}[theorem]{Definition}
\newtheorem*{definition*}{Definition}

\newtheorem*{observation*}{Observation}
\newtheorem{corollary}[theorem]{Corollary}
\newtheorem*{corollary*}{Corollary}
\newtheorem{lemma}[theorem]{Lemma \usefont{T1}{ptm}{m}{sl}}
\newtheorem*{lemma*}{Lemma}

\newtheorem{example}[theorem]{Example \usefont{T1}{ptm}{m}{sl}}
\newtheorem*{example*}{Example}

\newtheorem*{exercise*}{Exercise}

\newtheorem{remark}[theorem]{Remark \usefont{T1}{ptm}{m}{sl}}
\newtheorem*{remark*}{Remark}

\newtheorem*{note*}{Note}

\newtheorem*{question*}{Question}

\newtheorem*{notation*}{Notation}

\def\mc#1{\mathcal{#1}}

\def\ov{\overline}
\def\fill#1{\hbox to #1pt{\hfill}}
\def\hyphen{\,\text{-}\,}

\begin{document}
\title{Galois Theory of Hopf--Galois Extensions}
\authori{D. Marciniak}
\authoriaddress{}
\authoriemail{Dorofia@gmail.com}
\authorii{M. Szamotulski}
\authoriiaddress{}
\authoriiemail{Mszamot@gmail.com}
 

\maketitle
\thispagestyle{fancy}
\fancyfoot{}

\begin{abstract}   
	We introduce Galois Theory for Hopf--Galois Extensions proving
	existence of a Galois connection between subalgebras of an H-comodule
	algebra and generalised quotients of the Hopf algebra H. Moreover, we
	show that these quotients Q which define Q-Galois extension are the
	closed elements of our Galois connection. We discus the important
	problem of existence of a bijective correspondence between right
	ideals coideals and right coideal subalgebras of a Hopf algebra. For
	cleft extensions we characterise closed elements of the Galois
	connection as Hopf--Galois extensions.  We describe the relation of
	our results to the work of F.~van~Oystaeyen, Y.~Zhang and also to the
	results of P.~Schauenburg on biGalois extensions.
\end{abstract} 

We present a construction of a Galois connection between the
complete lattice of subalgebras of an $H$-comodule algebra $A$ and the complete
lattice of generalised quotients of a Hopf algebra $H$ (quotients by
coideals right ideals): 
\begin{theorem*}[Galois Theory for $H$-comodule algebras]
	For an $H$-comodule algebra $A$ over a field~$k$ there exists
	a \textsf{Galois connection}
	(Definition~\ref{defi:Galois-connection}): 
	\begin{equation}\label{eq:Main-Theorem}
		\Sub_{\textit{alg}}(A/A^{co\,H})\lgalois{}{\hspace{-.1cm}A^{co\,Q}\ellmap{}Q}\qquot(H)
	\end{equation} 
	\nopagebreak[4]
	where \(\qquot(H)=\{H/I:I\text{ - coideal right ideal of
	}H\}\).
\end{theorem*}
\noindent For short we will call this theorem '\textsf{Galois Theory}'. It is
a core of a Galois Theory and connects interesting areas: classical
\textit{Galois Theory for field extensions}, \textit{Hopf--Galois Theory},
which can be now formulated as a Galois Theory, and \textit{Lattice Theory}.
First important conclusion of this result is a bijection between
\textsf{closed elements}, i.e. elements which belong to the image of a Galois
connection. On the left hand side, \textsf{closed elements} are the extensions
of the form: $A/A^{co\,Q}$ for some $Q\in\qquot(H)$ (which we call
$H$-subextensions). Our aim is to describe closed elements of the right hand
side. This will answer the question: which quotients of $H$ classifies the
extensions of the form $A/A^{co\,Q}$? In Proposition~\ref{prop:closed} we show
that if $Q$ is such that $A^{co\,Q}\subseteq A$ is $Q$-Galois
(Definition~\ref{defi:H-ext}), then $Q$ is \textsf{closed}. From
both~\citet[Remark~1.2]{hs:normal-bases} and Proposition~\ref{prop:closed} it
follows that if \(B\in\Sub_\textit{alg}(A/A^{co\,H})\) and \(Q\in\qquot(H)\)
are such that:
\begin{enumerate}
    \item \(B\subseteq A^{co\,Q}\), and \(\can:A\otimes_BA\sir A\otimes Q\) is
	bijective,
    \item \(A\) is right or left faithfully flat over \(B\), 
\end{enumerate}
then \(B\) and \(Q\) are the corresponding closed elements.

We note that, the above Theorem can be extended to the case of $C$-comodule algebras
where \(C\) is a coalgebra with a group-like element.

We provide all the ingredients of the theory, namely: we prove that the
inclusion relation on the set of right ideals coideals defines a complete
lattice structure. We show that the set of subobjects, i.e. sub-Hopf algebras
and more generally subalgebras right coideals\footnote{Provided lattice
    theoretic results can be easily extended to the case of subalgebras left
    coideals and other mixed sub/quotient structures.} of a Hopf algebra form
a complete lattice.  When the Hopf algebra is finite dimensional then all the
lattices are algebraic and dually algebraic (Definition~\ref{defi:alglat}). 

Furthermore, we prove that in the case of the $H$-extension $k\subseteq H$ and
\emph{cleft extensions} \emph{$Q\in\qquot(H)$ is closed if and only if $A/A^{co
	Q}$ is $Q$-Galois} (Proposition~\ref{prop:closed-of-qquot} and
Theorem~\ref{thm:cleft-case}).  Thus in this case, the generalised quotients
\(Q\in\qquot(H)\) which are \(Q\)-Galois classifies $H$-subextensions.

This subject was already investigated
in~\cite{ps:hopf-bigalois,ps:gal-cor-hopf-bigal}
and~\cite{fo-yz:gal-cor-hopf-galois} but lack of explicit formulas in terms of
Hopf algebra structures led to difficulties to define the Galois connection
for the generality that we are dealing with. To overcome these difficulties we
use a new approach. We use Lattice Theory which provides an explicit formula
for a Galois connections between complete lattices in terms of the poset
structures (see~\eqref{eq:poset-formula}).


Let us briefly sketch other results that we prove. In finite dimensional case
we show the following
\begin{theorem*}[Finite Hopf--Galois Theory]
	Let $H$ be a finite dimensional Hopf algebra and $A/A^{co H}$ an
	$H$-Hopf--Galois extension. Then the Galois connection
	\eqref{eq:Main-Theorem} restricts to an \textsf{isomorphism}:
	\[\Sub_{\textit{H\text{-}ext}}(A/A^{co H})\simeq\qquot(H)\] 
	where the left hand side is the lattice of all $H$-extensions, i.e.
	the extensions of the form $A/A^{co Q}$ for some $Q\in\qquot(H)$.
\end{theorem*}
\noindent Thus the lattice $\qquot(H)$ \emph{classifies intermediate $H$-extensions} of
an $H$-Galois extension~$A/A^{co H}$.

From our Galois Theory we derive the Chase--Sweedler Theorem which lives on the
crossroads of Galois Theory for field extensions and Hopf--Galois Theory.
\noindent Also generalisation of Chase--Sweedler Theorem by~\citet[Theorem
2.3]{fo-yz:gal-cor-hopf-galois} to the case of a noncommutative $H$-module
algebra $A$ such that $A$ is $H^*$-Hopf--Galois over a finite dimensional Hopf
algebra $H$ follows from our theorem for finite Hopf--Galois extensions.


Next we apply our Galois Theory to the $H$-extension $k\subseteq H$. In this case
our Galois connection specifies to the following Galois correspondence (we
present here the left comodule version rather than right one as above): 
\begin{equation}\label{eq:Galois-connection-for-Hopf-algebra-left} 
\begin{array}{ccc}
    \Bigg\{K\subseteq H:K\text{\small - right coideal subalgebra}\Bigg\}&\hspace{-.4cm}\galois{\psi}{\phi}&\hspace{-.4cm}
    \Bigg\{H/I:I\text{ \small - left ideal coideal}\Bigg\}\\
\end{array}
\end{equation} 
\[\psi(K)=H/HK^+,\ \phi(H/I):=\,^{co\,H/I}H\] 
%
\noindent It was shown by~\citet[Theorem~3]{mt:rel-hopf-mod} in commutative
case and then proved by A.~Masuoka in the noncommutative
setting~\citeyearpar[Theorem~1.11]{am:quotient-theory-of-hopf-algebras} that there is
the
following bijective correspondence: \begin{equation}\label{eq:Takeuchi-prelim}
    \begin{array}{ccc} \Bigg\{K\subseteq H:K\text{-}\begin{array}{l}\text{
		    \small  right coideal subalgebra,}\\ H\text{ \small
		    faithfully flat over
		}K\end{array}\!\Bigg\}&\hspace{-.4cm}\galois{\psi}{\phi}&\hspace{-.4cm}
	\Bigg\{H/I:I\text{-}\begin{array}{l}\text{ \small left ideal
		    coideal,}\\ H\text{\small faithfully coflat}\\\text{
		    \small  over }H/I\end{array}\!\Bigg\}\\ \end{array}
\end{equation} Additionally H.-J.~Schneider proved that the above bijection
restricts to normal/conormal elements:
\begin{equation}\label{eq:Schneider-prelim} \begin{array}{ccc}
	\Bigg\{K\subseteq H:\!K\!\text{-}\!\begin{array}{l}\text{\small normal
		    sub-Hopf algebra,}\\ H\text{ \small faithfully flat over
		}K\end{array}\!\Bigg\}&\hspace{-.4cm}\galois{\psi}{\phi}&\hspace{-.4cm}
	\Bigg\{\!H/I:\!I\!\text{-}\begin{array}{l}\text{ \small  normal Hopf
		    ideal,}\\ H\text{ \small faithfully coflat}\\\text{ \small
		    over }H/I\end{array}\!\Bigg\}\\
\end{array} \end{equation} In this case we can sharpen our Galois Theory:
\begin{theorem*}\label{thm:closed}
	Let $H$ be a Hopf algebra. Then $Q\in\qquot(H)$ is \textsf{closed
	element} of the Galois connection~\eqref{eq:Main-Theorem} for 
	the $H$-extension $H/k$ if and only if $H/H^{co\,Q}$ is \textsf{$Q$-Galois}.
\end{theorem*}
\noindent There is the question of S.~Montgomery \emph{if the bijective
    correspondences~\eqref{eq:Takeuchi-prelim} and~\eqref{eq:Schneider-prelim}
    still survive if we don't assume faithfully flat/coflat conditions}
\cite[see][]{sm:hopf-alg}.  There is a positive answer to this question in the
case of finite dimensional Hopf algebras (Theorem~\ref{thm:newTakeuchi}) due
to~\citet[Corollary~6.5]{ss:projectivity-over-comodule-algebras}, who showed
that \(H\) a free module over any right coideal subalgebra, extending the
Nichols--Zoeller Theorem. Our results allow for a reformulation of the
infinite dimensional case:
\begin{proposition*}\label{prop:sm}
	Let \(H\) be a Hopf algebra. Then there is a bijective correspondence:
	\begin{equation}\label{eq:sm}
	    \bigg\{K\subseteq H:\,K\,-\text{right coideal subalgebra}\bigg\}\galois{\simeq}{}\bigg\{H/I:\,I\,-\text{left ideal coideal}\bigg\}
	\end{equation}
	if and only if 
	\begin{enumerate}[topsep=0pt,noitemsep]
	    \item for every its generalised quotient $Q$ the extension
		$^{co\,Q}H\subseteq H$ is $Q$-Galois
	    \item $\;^{co\,H/K^+H}\!H\subseteq K$ for every right coideal
		subalgebra $K$ of $H$.
	\end{enumerate}
\end{proposition*}
\noindent Note that if \(H\) is faithfully flat over \(K\) then
by~\cite[Remark~1.2]{hs:normal-bases} \(\;^{co\,H/K^+H}\!H=K\). The
correspondence~\eqref{eq:sm} can be intuitively understood as follows: let
$G=\Spec(H)$ be an affine group scheme. The set (of isomorphism classes) of
transitive $G$ sets is in bijection with subgroups of~$G$. Right coideal
subalgebras generalises transitive $G$-sets (in commutative case these are
affine quotients of $\Spec(H)$ which poses the natural action of $\Spec(H)$).
On the other side, subgroups correspond to quotients of the Hopf algebra $H$
(but we go beyond affine quotients as a general quotient might not be an
algebra).

F.~van~Oystaeyen and Y.~Zhang in~\cite{fo-yz:gal-cor-hopf-galois} prove
a noncommutative generalisation of the Chase--Sweedler theorem. In their paper
for the first time appear a remarkable construction of an additional Hopf
algebra. They construct a Hopf algebra $L(H,A)$ associated to
a commutative faithfully flat $H$-Hopf--Galois extension $A/B$ ($H$ is also
assumed to be commutative). The extension $A/B$ becomes
\emph{$L(H,A)$-$H$-bicomodule algebra} and a \emph{biGalois extension}. This
additional structure Hopf algebra $L(H,A)$ \emph{classifies intermediate
    $H$-comodule subalgebras} of $A/B$. Furthermore, when $A/B=\bE/\bF$ is
a field extension they prove the following Galois theorem:
\begin{theorem*}[{\citet[Theorem~4.7]{fo-yz:gal-cor-hopf-galois}}]
	Let $k\subseteq\bF$ be a field extension and let $H$ be a commutative
	and cocommutative $k$-Hopf algebra. Let $\bF\subseteq\bE$ be a
	field extension and an $H$-Hopf--Galois extension. Then there is
	a \textsf{one-to-one correspondence}:
	\begin{align}
		\left\{\begin{array}{c}
			\text{Hopf ideals of}\\
			\bF\otimes_kH\\
		\end{array}\right\}\ &\simeq\ 
		\left\{\begin{array}{c}
			H\text{-subcomodule}\\
			\text{subfields of }\bE\\
		\end{array}\right\}\label{eq:oz-1}
	\end{align}
	if $I$ -- a Hopf ideal, and $\bM$ -- an intermediate field extension of
	$\bE/\bF$, correspond to each other then\newline 
	\centerline{$\bE/\bM$ is $\bF\otimes_k(H/I)$-Hopf--Galois.} 
	Moreover, there is the following \textsf{bijection}:
	\begin{align}
		\left\{\begin{array}{c}
			\text{Hopf subalgebras of}\\
			\bF\otimes_kH\\
		\end{array}\right\}\ &\simeq\ 
		\left\{\begin{array}{c}
			$H$\text{-subcomodule}\\
			\text{subfields of }\bE\\
		\end{array}\right\}\label{eq:oz-2}
	\end{align}
	Furthermore, if $H'$ and $\bM$ correspond to each other then\newline 
	\centerline{$\bE/\bM$ is $(\bF\otimes_k H)/(\bF\otimes_k
	H)H'^+$-Hopf--Galois}
	where $H'^+:=\ker\epsilon\cap H'$.
\end{theorem*}
\noindent When $H$ is commutative and cocommutative then the Hopf algebra $L(H,A)$ is
equal to $\bF\otimes_kH$ \cite[Corollary~3.4]{fo-yz:gal-cor-hopf-galois}. The
proof of the previous theorem is based on this fact, so that $L(H,A)$ plays an
essential role. \cite{ps:hopf-bigalois} generalises the
construction of $L(H,A)$ to noncommutative extensions of rings $k\subseteq A$
over noncommutative (and noncocommutative) Hopf algebras. In his work
P.~Schauenburg proves the following theorem which is an extension of the
preceding result of F.~van~Oystaeyen and Y.~Zhang.
\begin{theorem*}[{\citet[Theorem~6.4]{ps:hopf-bigalois}}]
	Let $k\subseteq A$ be a faithfully flat $H$-Hopf--Galois extension of a
	ring $k$. Then there is the following \textsf{Galois connection}:
	\begin{equation}\label{eq:Schauenburg-Galois}
		\left\{\begin{array}{c}
			\text{coideals left ideals}\\
			\text{of }L(H,A)\\
		\end{array}\right\}\ \galois{}{}\ 
		\left\{\begin{array}{c}
			$H$\text{-subcomodule}\\
			\text{algebras of }A\\
		\end{array}\right\}
	\end{equation}
\end{theorem*}
\noindent If $B\in\Sub_{\textit{alg}^H}(A)$ is an $H$-subcomodule algebra such
that $A_B$ is a faithfully projective then it is a closed element
of~\eqref{eq:Schauenburg-Galois}. The closed elements of the left hand side
are the coideals left ideals which are $k$ direct summands of $L(H,A)$.
Furthermore, if $A$ is a skew field then the Galois
connection~\eqref{eq:Schauenburg-Galois} is an isomorphism.  Another result of
this type is given in~\cite[Theorem 3.6]{ps:gal-cor-hopf-bigal} where
P.~Schauenburg shows that the above Galois connection is a bijection on the
set of (left, right) admissible objects.  Where (right, left) admissibility is
the (right, left) faithfully flat/coflat condition
(Definition~\ref{defi:admissible}). In this paper we show a similar statement
for the Galois connection~\eqref{eq:Main-Theorem}. We show that the map
$Q\selmap{}A^{co\,Q}$ of the Galois connection~\eqref{eq:Main-Theorem} is
injective on the subset of (right, left) admissible objects provided $A/A^{co
    H}$ is a faithfully flat $H$-Hopf--Galois extension. The main result of
this section is Corollary~\ref{cor:equal} in which we conclude that (left,
right) admissible quotients of \(L(A,H)\) and \(H\) classifies the same
subalgebras of a \(k\)-algebra \(A\). 

We shall remark, that the Galois correspondences between posets of generalised
quotients of $L(H,A)$ and subalgebras of \(A\) is a special case of our Galois
Theory.

\section{Preliminaries}

\begin{definition}
	\textsf{Partially ordered} set, \textsf{poset} for short, is a set $P$
	together with an order relation $\preceq$ which is reflexive,
	transitive and antisymmetric.
\end{definition}

\begin{definition}[Galois connection]\label{defi:Galois-connection}
	Let $(P,\preceq)$ and $(Q,\leq)$ be two partially ordered sets.
	Antimonotonic morphisms of posets $\phi:P\mpr{}Q$ and $\psi:Q\mpr{}P$
	establishes a \textsf{Galois connection} if 
	\begin{equation}\label{eq:galprop}
		\mathop{\forall}\limits_{p\in P}\ p\preceq \psi\circ \phi(p)\
		and\ \mathop{\forall}\limits_{q\in Q}\ q\leq \phi\circ \psi(q)
	\end{equation} 
	We refer to this property as the \textsf{Galois property}. An element
	of $P$ (or $Q$) will be called \textsf{closed} if it is invariant
	under $\psi\phi$ ($\phi\psi$ respectively). Sets of closed elements
	will be denoted by $\ov P$ and $\ov Q$. Another name which appear in
	the literature for this notion is \textsf{Galois correspondence}.
\end{definition}

\begin{proposition}\label{prop:iso}
Let $P\galois{\phi}{\psi}Q$ be a Galois connection. Then the following
holds:
\begin{enumerate}[topsep=0pt,noitemsep]
\item $\ov P=\psi(Q)$ and $\ov Q=\phi(P)$
\item The restrictions $\phi|_{\ov P}$ and $\psi|_{\ov Q}$ are \textsf{inverse
	bijections} of $\ov P$ and $\ov Q$.
\item Map $\phi$ is \textsf{unique} in the sense that there exists only one Galois
	connection of the form $(\tilde\phi,\psi)$ for some map
	$\tilde\phi:P\mpr{}Q$, i.e. $\tilde\phi=\phi$. A similar statement
	holds for $\psi$. 
\item The map $\phi$ is \textsf{mono} (\textsf{onto}) if and only if the map $\psi$ is \textsf{onto}
	(\textsf{mono}). 
\item If one of the two maps is an isomorphism then the second is its inverse.
\end{enumerate}
\end{proposition}
A lattice is a poset in which there exists supremum and infimum of any two
elementary subset or equivalently of any  finite subset. A lattice can also be
defined as an algebraic structure which has two binary operations: join (an
abstract supremum of two elements) denoted by $\vee$ and meet (an abstract
infimum of two elements) denoted by $\wedge$. 
We refer the reader to~\cite{gg:lattice-theory} for the theory of lattices.
\begin{definition*}
	A lattice $(L,\vee,\wedge)$ is \textsf{complete} if for every $B\subseteq
	L$ there exists $\sup B$ and $\inf B$.
\end{definition*}
In a lattice \(L\) there exists arbitrary infima if and only if there are
arbitrary suprema. 
\begin{definition}\label{defi:compact}
	An element $z$ of a lattice $L$ is called \textsf{compact} if for any
	subset $S\subseteq L$ such that $z\leq\bigvee S$ there exists a finite
	subset $S_f$ of $S$ with the property $z\leq\bigvee S_f$. 
\end{definition}
\begin{definition}\label{defi:alglat}
	A lattice is \textsf{algebraic} if it is complete and every its element is a
	supremum of compact elements. A lattice is dually algebraic if its
	dual, i.e. the one with the dual order, is algebraic.
\end{definition}
It is a well known theorem of Universal Algebra that lattices of subalgebras
and lattices of congruences (quotient structures) of any algebraic structure
are algebraic. In particular, the lattices of sub-objects and quotient objects of
classical algebraic structures like groups, semi-groups, rings, modules, etc.
are algebraic.

\section{Lattices of substructures and quotient structures}

\begin{proposition}
	Let \mbox{$(C,\Delta,\epsilon)$} be a coalgebra,
	\mbox{$(B,m,u,\Delta,\epsilon)$} a bialgebra and
	\mbox{$(H,m,u,\Delta,\epsilon,S)$} a Hopf algebra, all over a field
	$k$. Then subcoalgebras of $C$ -- \mbox{$(\Sub(C),\subseteq)$},
	subbialgebras of $B$ -- \mbox{$(\Sub_{\textit{bi}}(B),\subseteq)$} and sub-Hopf
	algebras of a Hopf algebra $H$ -- \mbox{$(\Sub_{\textit{Hopf}}(H),\subseteq)$}
	are complete lattices which additionally are algebraic and dually
	algebraic when $C,B,H$ are finitely dimensional.
\end{proposition}
The proof is straightforward and will be omitted. In finite dimensional case
the lattices are algebraic and dually algebraic since every element of
a lattice of subspaces of a finite dimensional vector space is compact. 


Let $C$ be a coalgebra. Let us introduce standard notation: we let $\coId(C)$
denote the set of coideals of $C$, and $\coId_l(C),\ \coId_r(C)$ - the sets of
left, respectively right, coideals of $C$. They form complete lattice with
respect to the inclusion. If \(C\) is finite dimensional these lattices are
algebraic.

%

One can 'cogenerate' a coideal by a subset $Y$ of a coalgebra $C$. This is
defined as a join of all the coideals contained in $Y$, i.e. it is the largest
coideal contained in $Y$. We use this notion to define the meet operation in the
poset of coideals. This is dual to the case of algebras where the join is
defined as the ideal generated by the set-theoretic sum.

\begin{lemma}
	Let  $C$ be a coalgebra and $I_1,I_2$ two coideals. Then we have the
	following formulas for meet and join in the lattice of coideals of
	a coalgebra \(C\):
\begin{equation*}
I_1\vee I_2=I_1+I_2,\quad I_1\wedge I_2=+\{I\in \coId(C):\;I\subseteq I_1\cap I_2\}.
\end{equation*}
\end{lemma}
The proof is straightforward and is left to the reader.
%
As a direct consequence we get that for a bialgebra $B$ be or a Hopf algebra
$H$ the posets \( (\Id_{\textit{bi}}(B),\subseteq)\),
\((\Id_{\textit{Hopf}}(H),\subseteq) \) are complete lattices which in finite
dimensional case are algebraic and dually algebraic.


Let $C$ be a coalgebra, $B$ a bialgebra and $H$ a Hopf algebra. Then we let
use the following notation: \(\Quot(C)=\{C/I|\;I\text{ is a coideal}\}\),
\(\Quot(B)=\{B/I|\;I\text{ is a biideal}\}\), \(\Quot(H)=\{H/I|\;I\text{ is
	    a Hopf ideal}\}\).  We define $\qquot(H)$
as the set of all quotients of $H$ by \textsf{coideal right ideal} (quotient
as a coalgebra and a right $H$-module).
\[\qquot(H):=\{H/I:\;I\text{ coideal right ideal}\}\]
As a poset it is dually isomorphic to the poset of right ideals coideals
of~\(H\).
\begin{proposition}\label{prop:qqlat}
	The poset $(\qquot(H),\succeq)$ is a \textsf{complete lattice}. When $H$ is
	finite dimensional then this lattice is \textsf{algebraic} and
	\textsf{dually algebraic}.
\end{proposition}


\section{Galois connection in Hopf--Galois theory}

In this section we prove Galois Theory: existence of a Galois connection
between lattice of subalgebras and generalised quotients of a Hopf algebra in
the case of comodule algebras. In general we do not assume that $H$ is finite
dimensional. 
\begin{definition}
	Let $B\subseteq A$ be an extension of algebras then by
	\mbox{$\Sub_{\textit{alg}}(B\subseteq A)$} we denote the lattice of all
	subalgebras of $A$ which contains $B$.
\end{definition}

The defined lattice is an interval in the algebraic lattice of subalgebras of
$A$.

\begin{lemma}
	The lattice $\Sub_{\textit{alg}}(B\subseteq A)$ is an algebraic lattice.
\end{lemma}
\noindent The shortest argument is that $\Sub_{\textit{alg}}(B\subseteq A)$ is
a lattice of algebras of suitably defined algebra (in the sense of Universal
Algebra, which has operations of the algebra \(A\) and operations which comes
from \(B\)-module structure). Every lattice of subalgebra of an algebra (in
the sense of Universal Algebra) is algebraic, thus the lemma follows.


\begin{definition}
	Let $H$ be a Hopf algebra, and $A$ an algebra. $A$ is said to be
	a \textsf{comodule algebra} if it is an $H$-comodule, which structure
	map is a map of algebras, i.e. there is a \emph{coassociative algebra
	    map} $\delta:A\mpr{}A\otimes H$ compatible with the counit of $H$:
	\begin{center}
	\begin{tikzpicture}[>=angle 60,thick]
	\matrix[matrix of math nodes,row sep=1cm,column sep=1.5cm]{
  	|(A)| A 		& |(AH)| A\otimes H\\
  	|(HA)| A\otimes H 	& |(AHH)| A\otimes H\otimes H\\
	};
	\begin{scope}
		\draw[->] (A) -- node[above]{$\delta$}(AH);
		\draw[->] (A) -- node[left]{$\delta$}(HA);
		\draw[->] (AH) -- node[right]{$\delta\otimes\id$} (AHH);
		\draw[->] (HA) -- node[below]{$\id\otimes\Delta$} (AHH);
	\end{scope}
	\end{tikzpicture}
	\hspace{1cm}
	\begin{tikzpicture}[>=angle 60,thick]
	\matrix[matrix of nodes,row sep=1cm,column sep=1cm]{
  	|(A)|  $A$	&  |(AH)| $A\otimes H$\\
			&  |(B)| $A$\\
	};
	\begin{scope}
		\draw[->] (A) -- node[above]{$\delta$} (AH);
		\draw[->] (AH) -- node[right]{$\id\otimes\epsilon$} (B);
		\draw[double] (A) -- (B);
	\end{scope}
	\end{tikzpicture}
	\end{center}
	Above, $\Delta$ stands for the comultiplication of $H$ and $\epsilon$
	is the counit of $H$. The set 
	$$A^{co\,H}:=\{a\in A:\;\delta(a)=a\otimes 1\}$$
	is a subalgebra of $A$ and is called \textsf{subalgebra of
	coinvariants}. 
\end{definition}
\begin{definition}[$H$-extension and $Q$-Galois extension]\label{defi:H-ext}
	Let $H$ be a Hopf algebra, $Q\in\qquot(H)$ and let $A$ be an
	$H$-comodule algebra with the subalgebra of coinvariants $B$. An
	intermediate extension of the form $A^{co\,Q}\subseteq A$ will be
	called an \textsf{intermediate $H$-extension}. The poset of all 
	intermediate $H$-extensions will be denoted as
	$\Sub_{\textit{H\text{-}ext}}(B\subseteq A)$. An intermediate $H$-extension
	$A^{co\,Q}\subseteq A$ will be called \textsf{$Q$-Galois} if the
	canonical map:
	\[\can_Q:A\otimes_{A^{co\,Q}}A\mpr{}A\otimes Q\]	
	is a bijection. The subposet of $\Sub_{\textit{H\text{-}ext}}(B\subseteq A)$
	consisting of all $Q$-Galois extensions will be denoted by
	$\Sub_{\textit{Q\text{-}Galois}}(B\subseteq A)$.
\end{definition}
\begin{theorem}[Galois Theory for $H$-extensions]\label{thm:main}
	Let $H$ be a Hopf algebra over a field $k$ and $A$ be an $H$-comodule
	algebra then there exists a \textsf{Galois connection} between 
	complete lattices:
\begin{equation}\label{eq:main}
	\qquot(H)\galois{\phi}{\psi}\Sub_{\textit{alg}}(A/A^{co\,H})
\end{equation}
where $\phi(Q)=A^{co\,Q}$.
\end{theorem}
\noindent The map $\psi$ is unique what follows from general statement on
Galois connections (Proposition~\ref{prop:iso}~(3)).  Applying
Proposition~\ref{prop:iso}~(2) we obtain that $\phi(\qquot(H))$ and
$\psi(\Sub_{\textit{alg}}(A^{co\,H}\subseteq A))$ are dually isomorphic
posets. Note, that the above statement can be proved in more general
context of coactions of coalgebras. 


To prove Galois Theory we will use the following existence theorem
for Galois connections.
\begin{theorem}\label{thm:exist}
	Suppose that $P,Q$ are posets and moreover $P$ is complete. Then
	antimonotonic function $\phi:P\mpr{}Q$ is part of a Galois connection
	if and only if it reflects all suprema into infima.
\end{theorem}
The above theorem is a special case of the P.J.~Freyd characterisation theorem
of adjoint functors \cite[Theorem~2,
p.121]{smc:categories-for-the-working-mathematician}. 
The pair of maps $(\phi,\psi):P\galois{}{}Q$ is a Galois connection if an only if
$(\phi,\psi):P\galois{}{}Q^{op}$ is a covariant adjunction where the category
structure is given by: the relation $p'\geq p$ is thought as a morphism from $p'$ to $p$.

For more on Galois connections we refer to~\cite{rb-rc-jg:methods-in-prog},
where among other results the above theorem is proved\footnote{See 
\textsf{nLab} entry for the modern view on adjoint functor theorem:
http://ncatlab.org/nlab/show/adjoint+functor+theorem or the
classical~\cite[p. 95]{smc:categories-for-the-working-mathematician}.}. Now we are
ready to prove Theorem~\ref{thm:main}.

\begin{proof}
	By Theorem~\ref{thm:exist}, it is enough to show that the map
	$\phi:Q\elmapt{} A^{co\,Q}$ reverses suprema, i.e.
	\[A^{co\,\bigvee_{i\in I}\, Q_i}=\bigcap_{i\in I}\,A^{co\,Q_i}\]
	From the set of inequalities: $\bigvee_{i\in I}\,Q_i\geq Q_j\ (j\in I)$ it
	follows that 
	\[A^{co\,\bigvee_{i\in I}\, Q_i}\subseteq\bigcap_{i\in I}\,A^{co\,Q_i}\]
	Fix an element $a\in\bigcap_{i\in I}\,A^{co\,Q_i}$. We let $I_i$ denote the
	coideal and right ideal such that $Q_i=H/I_i$. Then we can write the following
	\[\forall_{i\in I}\;a\in A^{co\,Q_i}\ \Leftrightarrow\ \forall_{i\in I}\;\delta(a)-a\otimes 1\in A\otimes I_i\ \Leftrightarrow\ \delta(a)-a\otimes 1\in A\otimes\bigcap_{i\in I}\,I_i\ \Leftrightarrow\ a\in A^{co\,\bigvee_{i\in I} Q_i}\]
	The equivalence in the middle holds because $\bigcap_{i\in I}A\otimes
	I_i=A\otimes \bigcap_{i\in I} I_i$ what we show below:
	\begin{align}
	\bigcap_{i\in I}\;A\otimes I_i\ni\sum_{l=1}^n a_l\otimes b_l&\;\Leftrightarrow\;\forall_{i\in I}\ \sum_{l=1}^n a_l\otimes b_l\in A\otimes I_i\notag\\
	&\;\Leftrightarrow\;\forall_{i\in I}\ \forall_{l=1,\dots,n}\ b_l\in I_i\notag\\
	&\;\Leftrightarrow\;\sum_{l=1}^n a_l\otimes b_l\in A\otimes\bigcap_{i\in I}I_i\notag
	\end{align}
	It remains to show that if $\delta(a)-a\otimes 1\in
	A\otimes\bigcap_{i\in I}I_i$ then \mbox{$\delta(a)-a\otimes 1\in
	    A\otimes\bigwedge_{i\in I}I_i$.} We proceed in three steps, first
	we prove this for $A=H$, then for $A\otimes H$ and then for any
	$H$-comodule algebra $A$. For $A=H$ this is equivalent to the
	existence of the Galois
	connection~\eqref{eq:Galois-connection-for-Hopf-algebra-left}. The case
	$A\otimes H$: let $x=\sum_{k=1}^na_k\otimes h_k\in A\otimes H$ be such
	that $\sum_{k=1}^na_k\otimes\Delta(h_k)- \sum_{k=1}^na_k\otimes
	h_k\otimes 1_H\in A\otimes H\otimes \bigcap_{i\in I}I_i$. Then we can
	choose $a_i$ such that they are linearly independent. Then by a choice
	of a complement of the span of $\{a_i\}_{i=1,\dots,n}$ we get
	$a_i^*:A\sir k$ such that $a_i^*(a_j)=0$ for $i\neq j$ and
	$a_i^*(a_i)=1$ for $i=1,\dots,n$.  Using $a_i^*$ we can reduce the
	question to the previous case. For general case, observe that if
	$a_{\mathit{(0)}}\otimes a_{\mathit{(1)}}-a\otimes 1\in A\otimes\bigcap_{i\in I}I_i$
	then $a_{\mathit{(0)}}\otimes a_{\mathit{(1)}}\otimes a_{\mathit{(2)}}-a_{\mathit{(0)}}\otimes
	a_{\mathit{(1)}}\otimes 1\in A\otimes H\otimes\bigcap_{i\in I}I_i$, thus by the
	previous case $a_{\mathit{(0)}}\otimes a_{\mathit{(1)}}\otimes a_{\mathit{(2)}}-a_{\mathit{(0)}}\otimes
	a_{\mathit{(1)}}\otimes 1\in A\otimes H\otimes\bigwedge_{i\in I}I_i$.
	Computing $\id_A\otimes\epsilon\otimes\id_H$ we get
	$\delta(a)-a\otimes 1\in A\otimes\bigwedge_{i\in I}I_i$.
\end{proof}
Note that the map: \(\psi\) of the above Galois connection may have be
defined, by the following formula:
\begin{equation}\label{eq:poset-formula}
    \psi(A')=\bigvee\{Q\in\qquot(H): A'\subseteq A^{co\,Q}\}
\end{equation}

\section{Closed elements of Galois connection for Hopf--Galois extensions}

In this section we show which elements of $\qquot(H)$ are closed in the Galois
connection~\eqref{eq:main}. The importance of this theorem lies in the fact
that closed elements of $\qquot(H)$ classifies elements of
$\Sub_\textit{H\text{-}ext}(A/B)$. Our main result of this section states that
$Q\in\qquot(H)$ is closed whenever $A/A^{co\,Q}$ is $Q$-Galois. It follows
that if $H$ is finite dimensional then every generalised quotient is closed
provided $A/A^{co\,H}$ has surjective (thus bijective) canonical map. As
a corollary we obtain a bijective correspondence between
$\Sub_\textit{H\text{-}ext}(A/B)$ and $\qquot(H)$ for finite dimensional Hopf
algebras.

\begin{proposition}\label{prop:mono}\footnote{We would like to thank P.~Hajac
	for his insight which helped to prove this proposition.} 
	Let $A$ be an $H$-comodule algebra (both $A$ and $H$ can be infinite
	dimensional) with surjective canonical map and let $A$ be
	a $Q_1$-Galois and a $Q_2$-Galois extension where $Q_1,Q_2\in
	\qquot(H)$. Then the following implication holds: 
	    \[A^{co\,Q_1}=A^{co\,Q_2}\ \Rightarrow\ Q_1=Q_2\]
\end{proposition}

\begin{proof}
Let $B=A^{co\,Q_1}=A^{co\,Q_2}$ then we have the following commutative diagram:
\begin{center}
\begin{tikzpicture}[>=angle 60,thick]
\matrix[matrix of math nodes,column sep=1cm,row sep=1cm]{
		    &					& |(AQ1)| A\otimes Q_1\\
|(A)| A\otimes_B A  & |(B)| A\otimes_{A^{co\,H}}A	& |(C)|   A\otimes H\\
		    &					& |(AQ2)| A\otimes Q_2\\ 
};
\begin{scope}
\draw[->] (A) -- node[above]{$\can_{Q_1}$} (AQ1);
\draw[->] (A) -- node[below]{$\can_{Q_2}$} (AQ2);
\draw[<<-] (A) -- (B);
\draw[->] (B) -- node[above,pos=.4]{$\can$}(C);
\draw[->] (C) -- node[right]{$id\otimes\pi_1$}(AQ1);
\draw[->] (C) -- node[right]{$id\otimes\pi_2$}(AQ2);
\end{scope}
\end{tikzpicture}
\end{center}
The maps $\can_{Q_1}$ and $\can_{Q_2}$ are isomorphisms. Let 
$f:=(\can_{Q_1}\circ\can_{Q_2}^{-1})\circ(id\otimes\pi_2)$, $g:=(id\otimes\pi_1)$.
By commutativity of the above diagram, $f\circ\can$ and $g\circ\can$ are
equal. Moreover, surjectivity of $\can$ yields
the equality $(\can_{Q_1}\circ
\can_{Q_2}^{-1})\circ(id\otimes\pi_2)=(id\otimes\pi_1)$. It follows that there
exists $\pi:\,Q_1\ir Q_2$ such that $\can_1\circ \can_2^{-1}=id\otimes \pi$
and $\pi\circ\pi_2=\pi_1$.  Furthermore, $\pi$ is right $H$-linear and
colinear, thus $Q_2\succeq Q_1$.
%
In the same way we prove that $Q_1\succeq Q_2$ (take $\can_2\circ
\can_1^{-1}$ instead of $\can_1\circ \can_2^{-1}$); because $\succeq$ is an
order (antisymmetry) we get $Q_1=Q_2$.
\end{proof}

%
%

\begin{proposition}\label{prop:closed}
	If $A$ is an $H$-comodule algebra with epimorphic canonical map
	(both $A$ and $H$ can be infinite dimensional) then every $Q\in
	\qquot(H)$ for which $A^{co\,Q}\subseteq A$ is a $Q$-Galois extension
	is a \textsf{closed element} of the Galois connection~\eqref{eq:main}.
\end{proposition}
\begin{proof}
Fix $A^{co Q}$ for some $Q\in \qquot(H)$. Then $\phi^{-1}(A^{co\,Q})$ is
a non-empty upper-sublattice  of $\qquot(H)$ (i.e. it is a subposet closed
under finite suprema) which has the greatest element, namely $\widetilde
Q=\psi(A^{co\,Q})$. This follows from Theorem~\ref{thm:main} and properties of
Galois connections. The generalised quotient $\widetilde Q$ is the only closed
element belonging to \mbox{$\phi^{-1}(A^{co\,Q})$}. Both $Q\leq\psi(A^{co\,Q})$ and
the observation that $A^{co\,Q}\subseteq A$ is $Q$-Galois imply that
$\widetilde Q$ is also such.  We have the commutative diagram:
\begin{center}
\begin{tikzpicture}
	\node (A) at (0,0) 	{$A\otimes_B A$};	         	
	\node (B) at (3cm,0)  	{$A\otimes H$};
  	\node (C) at (0cm,-1.7cm) {$A\otimes_{A^{co\,\widetilde Q}}A$}; 
	\node (D) at (3cm,-1.7cm) {$A\otimes \widetilde Q$};
  	\node (E) at (0cm,-3.4cm) {$A\otimes_{A^{co\,Q}}A$};		
	\node (F) at (3cm,-3.4cm) {$A\otimes Q$};
\begin{scope}[>=angle 60,thick]
\draw[->>] (A) -- node[above]{$\can_H$} (B);
\draw[->>] (A) -- (C);
\draw[->>] (B) -- (D);
\draw[->]  (C) -- node[above]{$\can_{\widetilde Q}$} (D);
\draw[->]  (E) -- node[below]{$\simeq$} node [above]{$\can_Q$}  (F);
\draw[->]  (C) -- node[left]{$=$}(E);
\draw[->>] (D) -- (F);
\end{scope}
\end{tikzpicture}
\end{center}
From the lower commutative square we get that $\can_{\widetilde Q}$ is a
monomorphism and from the upper commutative square one can deduce that
$\can_{\widetilde Q}$ is onto. Unless $\widetilde Q=Q$ we get a contradiction
with Proposition~\ref{prop:mono}.
\end{proof}

\begin{corollary}
    Let \(B\in\Sub_\textit{alg}(A/A^{co\,H})\) and \(Q\in\qquot(H)\) be such that:
    \begin{enumerate}
	\item \(B\subseteq A^{co\,Q}\), and \(\can:A\otimes_BA\sir A\otimes Q\) is
	    bijective,
	\item \(A\) is right or left faithfully flat over \(B\), 
    \end{enumerate}
    then, by~\cite[Remark~1.2]{hs:normal-bases} \(B=A^{co\,Q}\) and thus
    \(A/A^{co\,Q}\) is a \(Q\)-Galois extension. It follows that \(B\) and
    \(Q\) are the corresponding closed elements.
\end{corollary}

\citet[Corollary~3.3]{ps-hs:gen-hopf-galois} give conditions under which an
extension $^{co\,Q}A\subseteq A$ is a $Q$-Galois and thus when it is a closed
element of the Galois connection~\eqref{eq:main}. Let us cite it
here:

\begin{proposition}[{\citet[Corollary~3.3]{ps-hs:gen-hopf-galois}}]\label{prop:corollary-3.3}
	Let $H$ be a Hopf algebra over a ring $k$ with bijective antipode and
	let $A$ be an $H$-comodule algebra with epimorphic canonical map. Let
	$Q\in\qquot(H)$, then it follows that in each of the case $A/A^{co\,Q}$
	is $Q$-Galois and $A$ is a projective left $A^{co\,Q}$-module:
	\begin{enumerate}
		\item[(1)] $k$ is a field and $H$ is finite dimensional,
		\item[(2)] $H$ is finitely generated projective over k, coflat
		    as a right Q-comodule, and the surjection $H\mpr{}Q$
		    splits as a left $Q$-comodule map,
		\item[(3)] $H$ has enough right integrals, is coflat as
		    a right $Q$-comodule, and the surjection $H\mpr{}Q$ splits
		    as a left $Q$-comodule map,
		\item[(4)] $k$ is a field, $H$ is co-Frobenius, and faithfully
		    coflat both as a left and a right $Q$-comodule,
		\item[(5)] $H$ is $Q$-cleft and $Q$ is finitely generated
		    projective,
		\item[(6)] $k$ is a field, $H$ has cocommutative coradical,
		    and $Q$ is finite dimensional and of the form $Q=H/K^+H$
		    for a Hopf subalgebra $K$ of $H$.
	\end{enumerate}
\end{proposition}
Let us remark that due to~\citet{rl-ms:bilinear-form-for-hopf-algebras} finite
dimensional Hopf algebras have bijective antipode and thus it is not needed to
assume it in point (1) above.  The following theorem is a generalisation of
\cite[Theorem~4.7]{fo-yz:gal-cor-hopf-galois} which we included in the
introduction.
\begin{theorem}\label{thm:f-dim}
	Let $H$ be a finite dimensional Hopf algebra and let $B\subseteq A$ be
	an $H$-Hopf--Galois extension. Then every intermediate $H$-extension is
	a $Q$-Galois and there is an anti-isomorphism of posets: 
	\[\Sub_{\textit{H\text{-}ext}}(B\subseteq A)\simeq\qquot(H)\] 
\end{theorem}
\begin{proof}
	It follows directly from Proposition~\ref{prop:closed}
	Proposition~\ref{prop:corollary-3.3} (1) and
	Proposition~\ref{prop:iso}.
\end{proof}
\begin{remark}
	It follows that intermediate $H$-extensions form a complete lattice.
	Moreover, the poset of Hopf--Galois subextensions is dually isomorphic
	to the lattice of Hopf algebra quotients of $H$:
	\[\left\{A^{co\,H/I}\subseteq A:\,I\,-\,\text{Hopf ideal of }H\right\}\simeq\Quot(H).\] 
\end{remark}

This corollary implies part of the Galois Theory for finite field extensions
(the injectivity of the map $\Fix:\Sub(\Gal(\bE/\bF))\ir\Sub(\bE/\bF)$). The
main result in finite Hopf--Galois theory -- Theorem~\ref{thm:f-dim}, shows
that there is bijective correspondence with all $H$-subextensions. The map
$\psi$ constructed as an adjunction of $Q\mathop{\mapsto}\limits^\phi A^{co
    Q}$ is not given by an explicit formula in terms of the Hopf algebra
structure (it can be expressed using the order structure,
see~\eqref{eq:poset-formula}). However, in some important cases: the finite
Galois theory, the $k\subseteq H$ $H$-Hopf--Galois extension and cleft
extensions we will see an explicit Hopf algebraic formula.

In \cite[Theorem 2.3]{fo-yz:gal-cor-hopf-galois} there is proved an extension
of Chase--Sweedler theorem for commutative algebras which are comodule algebras
over finite dimensional Hopf algebra. Theorem~\ref{thm:f-dim} extends this
result to non-commutative algebras.

\section{Chase--Sweedler Theorem}
As a direct consequence of Theorem~\ref{thm:main} we get the Chase--Sweedler Theorem
for finite dimensional Hopf algebras. 
\begin{definition}
Let $H$ be a Hopf algebra over $k$. A $k$-algebra $A$ is an $H$-module algebra
if it is an $H$-module satisfying the following identity:
\[ 
h(a_1a_2)=h_{\mathit{(0)}}(a_1)h_{\mathit{(1)}}(a_2)
\] 
The subalgebra of invariants is defined as $A^H:=\{a\in A:\forall_{h\in H}
    ha=\epsilon(h)a\}$.
\end{definition}
 
Let $A/B$ be an extension of \(k\)-algebras and $A$ be an $H$-module algebra.
It will be called Hopf--Galois extension if and only if the canonical map
\[ 
\can:A\otimes_BA\mpr{}\Hom(H,A),\;a_1\otimes_B a_2\mapsto (h\mapsto a_1ha_2)
\] 
is a bijection.
\begin{theorem*}[{\cite[Chase--Sweedler]{sc-ms:hopf-algebras-and-galois-theory}}]
Let $A/B$ be a Galois extension such that it is a Hopf--Galois extension
under an action of a finite cocommutative Hopf algebra $H$. Then the following map is
injective and inclusion reversing:
\[ \Sub_{\textit{Hopf}}(H)\mpr{}\Sub_{\textit{field}}(A\subseteq B),\ H\selmap{} A^H \] 
\end{theorem*}
\noindent This theorem follows from Theorem~\ref{thm:f-dim} and the bijective
correspondence between right $H$-comodule structures and left
$H^*$-module structures on $A$, assuming that $H$ is finitely dimensional (see
\cite[Proposition~6.2.4]{sd-cn-sr:hopf-alg}). Note that Chase and Sweedler
proved this for algebras over rings. Furthermore, they classified
\(H\)-extensions.

In \cite[Theorem 2.3]{fo-yz:gal-cor-hopf-galois} there is proved an extension
of Chase--Sweedler theorem for commutative algebras which are comodule algebras
over finite dimensional Hopf algebra. Theorem~\ref{thm:f-dim} also extends this
result to non-commutative algebras.

\section{The \(H/k\)-\(H\)-extension}
In the case of an $H$-Hopf--Galois extension $k\subseteq H$ we show that
Theorem~\ref{thm:main} generalises
M.~Takeuchi's~\citeyearpar[Theorem~3]{mt:rel-hopf-mod},
A.~Masuoka's~\citeyearpar[Theorem~1.11]{am:quotient-theory-of-hopf-algebras} and
H.-J.~Schneider's~\citeyearpar[Theorem~1.4]{hs:exact-seq-qg} results.  They provide
a bijection between some (normal) Hopf ideals and some (conormal) Hopf
subalgebras of a given Hopf algebra. Below we cite these theorems merged
together. H.-J.~Schneider added the normality/conormality condition to the
A.~Masuoka result, who generalised the commutative case considered by
M.~Takeuchi. The definition of normal Hopf ideals and normal Hopf
subalgebras can be found in~\cite[Definition~1.1]{hs:exact-seq-qg}.
\begin{theorem}\label{thm:Takeuchi}
Let $H$ be a Hopf algebra. Then
\begin{equation}\label{eq:Takeuchi}
\begin{array}{ccc}
	\Bigg\{K\subseteq H:\,K\,\hyphen\begin{array}{l}\text{\small right coideal}\\\text{\small subalgebra}\\ H\ \text{\small f. flat over }K\end{array}\!\Bigg\}&\hspace{-.4cm}\galois{\psi}{\phi}&\hspace{-.4cm}\Bigg\{H/I:\,I\,\hyphen\begin{array}{l}\text{\small left ideal coideal}\\ H\ \text{\small f. coflat over }H/I\end{array}\!\Bigg\}\\
\end{array}
\end{equation}
\[\psi(K)=H/HK^+,\ \phi(H/I):=\,^{co\,H/I}H\] 
are inverse bijections. They restricts to normal/conormal elements: 
\begin{equation}\label{eq:Schneider} 
\begin{array}{ccc}
	\bigg\{K\subseteq H:\,K\,\hyphen\begin{array}{r}\text{\small normal sub-Hopf}\\\text{\small algebra}\\ H\ \text{\small f. flat over }K\end{array}\!\bigg\}&\hspace{-.4cm}\simeq&\hspace{-.4cm}\bigg\{H/I:\,I\,\hyphen\begin{array}{l}\text{\small normal Hopf ideal}\\ H\ \text{\small f. coflat over }H/I\end{array}\!\bigg\}\\
\end{array}
\end{equation}
\end{theorem}
\noindent In \cite[Theorem~3.10]{ps:gal-cor-hopf-bigal} shows that for
a $k$-flat Hopf algebra $H$ the above bijective correspondence restricts to
(left, right) admissible objects (Definition~\ref{defi:admissible}) of right
and left hand sides.  We present the previous theorem in the same way as it
was originally stated, however in this paper we work in a dual setting than
H.-J.~Schneider and M.~Takeuchi:  
\begin{center}
\mbox{right/left coideal subalgebras $\galois{}{}$ quotients left/right module coalgebras}
\end{center}
All the results are true in both cases. We switch to the convention of
H.-J.~Schneider and M.~Takeuchi.
\begin{definition}
We let $\qsub(H)$ denote the poset of right coideals subalgebras of a Hopf algebra
$H$.
\end{definition}
\noindent If $H$ is commutative then right coideal subalgebras over which $H$ is
faithfully flat correspond to quotients of $\Spec(H)$ by an affine closed
subgroup scheme (thus the quotient is affine and a transitive $\Spec(H)$-set).
The poset $\qsub(H)$ has all infima, hence it has a unique structure of
a complete lattice. 
\begin{proposition}[{\citet[Proposition~1]{mt:rel-hopf-mod}}]\label{prop:Takeuchi}
	Let $H$ be a Hopf algebra and $I$ its left ideal coideal. Then
	$^{co\,H/I}H$ is a right coideal subalgebra of $H$. Let $K$ be a right
	coideal subalgebra of $H$. Then $HK^+$ is a left ideal coideal of $H$.
\end{proposition}


The following example is known (with the exception of point 2), but the
arguments are spread in the literature, thus we explain it more carefully.
\begin{example}\label{ex:1}
	Let $H$ be a Hopf algebra (possibly infinite dimensional) and let $K$
	be its right coideal subalgebra. Then $^{co\,H/HK^+}\!H\subseteq H$ is
	$H/HK^+$-Galois. The inverse of 
$$\can:H\otimes_{^{co\,H/K^+H}H}H\lmpr{}H/HK^+\otimes H,\ x\otimes y\elmap{} \ov x_{\mathit{(1)}}\otimes x_{\mathit{(2)}}y$$
	is given by: $\can^{-1}(\ov x\otimes y)=x_{\mathit{(1)}}\otimes
	S(x_{\mathit{(2)}})y$, where $\ov x$ is the class of $x$ in $H/HK^+$.
	This map is well defined since $K$ is a subcoalgebra: if $x=hk$, where
	$k\in K^+$ and $h\in H$, then
$$\begin{array}{cll}
      can^{-1}(\ov{hk}\otimes y)&=h_{\mathit{(1)}}k_{\mathit{(1)}}\otimes_{}S(k_{\mathit{(2)}})S(h_{\mathit{(2)}})y&\\
      &=h_{\mathit{(1)}}\otimes k_{\mathit{(1)}}S(k_{\mathit{(2)}})S(h_{\mathit{(2)}})y&\text{since }K\subseteq\, ^{co\,H/HK^+}\!H\\
				&=h_{\mathit{(1)}}\otimes\epsilon(k)S(h_{\mathit{(2)}})y\ =\ 0&\text{since }k\in \ker\,\epsilon
\end{array}$$
	\noindent where the tensor on the left hand side is over
	$\;^{co\,H/HK^+}\!H$.  Let us show that in fact
	\mbox{$K\subseteq\,^{co\,H/HK^+}\!H$:}
\begin{equation}\label{eq:1}
	\delta(k)=\ov k_{\mathit{(1)}}\otimes k_{\mathit{(2)}}=\ov 1\epsilon(k_{\textit{(1)}})\otimes k_{\mathit{(2)}}=\ov 1\otimes k
\end{equation}
	\noindent The second equality holds since $k_{\mathit{(1)}}\in K$ and
	the two maps $K\lmprr{\pi|_K}{k\mapsto\ov 1\epsilon(k)}H/HK^+$ have
	the same kernels equal $K^+$. Furthermore, $K=k1+K^+$ and both maps
	have the same value on $1$ thus they are equal. Moreover,
	$K=\,^{co\,H/HK^+}\!H$ in the following two cases:
\begin{enumerate}
    \item $H$ is left or right faithfully flat over $K$~\cite[see][Remark~1.2]{hs:normal-bases},
%
	\item $H$ is finite dimensional:
	    by~\cite[Thm.~6.1(ii)]{ss:projectivity-over-comodule-algebras}
	    \(H\) is free over \(K\) thus faithfully flat, now use~\((1)\).
	\item Moreover, in these two cases, if $K$ is a normal
		Hopf subalgebra then $HK^+=K^+H$ is a normal Hopf ideal, and if
		$I$ is a normal Hopf ideal then $^{co\,H/I}\!H=H^{co\,H/I}$ is
		a normal Hopf subalgebra what was originally proved by
		H.-J.~Schneider in~\cite[Lemma~1.3]{hs:exact-seq-qg} to
		show~\eqref{eq:Schneider}. 
\end{enumerate}
\end{example}

In the second case flatness is necessary as it shows the following example
which we cite after~\cite{ss:projectivity-over-comodule-algebras} (but used
rather for different purposes):
\begin{example}
	Let $F_1$ be the free group with one generator $g$ and $M_1$ its free
	submonoid generated by $g$. Then the Hopf algebra $H=k[F_1]$ and its
	coideal subalgebra $K=k[M_1]$ are such that $HK^+=HH^+$, then the
	above implies that \(H=K\) which is false but $k[F_1]$ is not
	faithfully flat over $k[M_1]$: $k[\bZ_n]\otimes_{k[M_1]}k[F_1]=0$,
	where $\bZ_n$ is the group of integers modulo $n$.
\end{example}

The two cases of the preceding Theorem lead us to two new results which
positively answers the question raised by~\citet{sm:hopf-alg}:
\emph{is correspondence~\eqref{eq:Takeuchi} a bijection without extra
    assumptions?} The above example gives a negative answer to this question,
however if \(H\) is finite dimensional over a field then~\eqref{eq:Takeuchi}
is a bijection without faithful flat/coflat assumptions, because they are
satisfied for every subobject.

\begin{theorem}\label{thm:newTakeuchi}
	Let $H$ be a Hopf algebra over a field $k$. Then $k\subseteq H$ is an
	$H$-Hopf--Galois extension and there exists a \textsf{Galois
	connection}:
	\begin{equation}\label{eq:galois-for-hopf-alg}
	    \begin{array}{ccc}
		\bigg\{K\subseteq H:\,K\,\hyphen\text{right coideal subalgebra}\bigg\}&\hspace{-.3cm}\galois{\psi}{\phi}&\hspace{-.3cm}\bigg\{H/I:\,I\,\hyphen\text{left ideal coideal}\bigg\}\\[-2mm]
		=:\qsub(H)&&=:\qquot(H)
	    \end{array}
	\end{equation}
	where $(\phi(Q)=\,^{co\,Q}H,\psi(K)=H/HK^+)$ is the Galois
	connection\footnote{Here we changed one side of the Galois connection
	comparing to Theorem~\ref{thm:main} but it doesn't make a difference.}
	obtained in Theorem~\ref{thm:main}. Moreover, this Galois
	correspondence restricts to normal elements:
	\[\begin{array}{ccc}
	    \bigg\{K\subseteq H:\,K\,\hyphen\text{normal Hopf subalgebra}\bigg\}&\hspace{-.3cm}\galois{\psi}{\phi}&\hspace{-.3cm}\bigg\{H/I:\,I\,\hyphen\text{normal Hopf ideal}\bigg\}\\[-2mm] 
	    =:\Sub_{\textit{nHopf}}(H)&&=:\Quot_{\textit{normal}}(H)
	\end{array}\]
We claim that:
\begin{enumerate}
	\item[(1)] $K\in\qsub(H)$ such, that $H$ is faithfully flat over $K$,
	    is a closed element of the Galois
	    connection~\eqref{eq:galois-for-hopf-alg}.
	\item[(2)] $Q\in\qquot(H)$ such that $H$ is faithfully coflat over $Q$
	    is a closed element of the Galois
	    connection~\eqref{eq:galois-for-hopf-alg}.
	\item[(3)] if $H$ is finite dimensional then $\phi$ and $\psi$ are
	    \textsf{inverse bijections}.
\end{enumerate}
\end{theorem}

The above theorem we put it here for the sake of completeness. In the proof we
use Proposition~\ref{prop:mono}, which might be avoided using P.Schauenburg
results (see~\cite[Theorem~3.10]{ps:gal-cor-hopf-bigal}). Point~(1) gives an
alternative proof of the A.~Masuoka's part of Theorem~\ref{thm:Takeuchi}
observing that $H$ is faithfully flat over a right coideal subalgebra $K$ if
and only if $H$ is faithfully coflat over $H/HK^+$ (for Hopf algebra $H$ over
a field $k$). We refer to~\cite[Proposition 4.5]{ps-hs:gen-hopf-galois} for
the proof of this observation. Point~(3) uses the S.~Skryabin
result~\cite[Theorem~6.1]{ss:projectivity-over-comodule-algebras}. We will
also see that the presented method will shed more light in the infinite
dimensional case of S.~Montgomery question.

\begin{proof}
	Proposition~\ref{prop:Takeuchi} shows that both maps $\phi$ and $\psi$
	are well defined. Equation~\eqref{eq:1} shows that
	$K\subseteq\phi\psi(K)=\,^{co\,H/HK^+}\!H$ thus to obtain that $(\phi,
	\psi)$ is a Galois connection it remains to prove that
	$H/I\leq\psi\phi(H/I)=H/H(^{co\,H/I}\!H)^+$, i.e.  $I\supseteq
	H(^{co\,H/I}\!H)^+$. Let $x\in (^{co\,H/I}\!H)^+$ then 
	\[\Delta(x)=x_{\mathit{(1)}}\otimes x_{\mathit{(2)}}=1\otimes x\;+\;\sum_k\,i_k\otimes x_k\]
	where $i_k\in I,\; x_k\in H,\;k=1,\dots,n$, thus 
	\begin{align*}
		x\;	&=\;1\epsilon(x)\;+\;\sum_k\,i_k\epsilon(x_k)	&\\
			&=\;\sum_k\,i_k\epsilon(x_k)\in I		&\text{since }x\in \ker\epsilon
	\end{align*}
	This Galois connection is the same as~\eqref{eq:main} in
	Theorem~\ref{thm:main}, because of the uniqueness of Galois maps
	(Proposition~\ref{prop:iso}~(3)). The minor difference is the codomain
	of $\phi$: here it is  $\qsub(H)$ rather than
	$\Sub_{\textit{alg}}(k\subseteq H)$ as in Theorem~\ref{thm:main}
	according to the case $A=H$. The map $\phi$ restricts to normal
	elements as shown by~\cite{hs:exact-seq-qg}. In the case of finite
	dimensional Hopf algebras every $Q$-extension $H^{co\,Q}\subseteq H$
	is $Q$-Galois (Proposition~\ref{prop:corollary-3.3}). From
	Proposition~\ref{prop:closed} we get that $\phi$ is a monomorphism and
	so by general properties of Galois connections:
	\[\psi\circ\phi(Q)=H/H(^{co\,Q}H)^+=Q\]
	\noindent Moreover, in any of the two cases: $H$ is finite dimensional
	or faithfully flat over $K$, we have
	\[\phi\circ\psi(K)=\,^{co\,H/HK^+}\!H=K\]
	\noindent Thus in fact $K$ is closed element of the Galois connection.
	Point (2) follows from Theorem~\ref{thm:Takeuchi}.
\end{proof}

In this setting we can prove inverse of Proposition~\ref{prop:closed}, and thus
obtain a full characterisation of closed elements of $\qquot(H)$.
\begin{proposition}\label{prop:closed-of-qquot}
	Let $H$ be a Hopf algebra. Then $Q\in\qquot(H)$ is a \textsf{closed
	element} of the Galois connection~\eqref{eq:galois-for-hopf-alg} if and
	only if $H/\,^{co\,Q}H$ is a Hopf--Galois extension.
\end{proposition}
\begin{proof}
	It is enough to show that if $Q$ is closed then $^{co Q}H\subseteq H$
	is an $H$-Hopf--Galois. If $Q$ is closed then $Q=H/H(^{co Q}H)^+$. One
	can show that for any $K\in\qsub(H)$ the following map is an
	isomorphism:
	\begin{equation}\label{eq:canK}
 		H\otimes_{K}H\mpr{}H/HK^+\otimes H,\quad h\otimes k\elmap{}h_{\mathit{(1)}}\otimes h_{\mathit{(2)}}k
	\end{equation}
	Its inverse is given by $H/HK^+\otimes H\ni\overline h\otimes
	k\,\elmap{}\,h_{\mathit{(1)}}\otimes S(h_{\mathit{(2)}})k\in
	H\otimes_K H$ which is well defined because $\Delta(K)\subseteq
	K\otimes H$. Plugging $K=\,^{co\,Q}H$ to equation~\eqref{eq:canK} we
	observe that this map is the canonical map of $Q$. Thus
	$H/\,^{co\,Q}H$ is a $Q$-Galois extension.
\end{proof}
Combining the above, Proposition~\ref{prop:mono} and
Proposition~\ref{prop:iso}(4) we get: 
\begin{corollary}
	Let $H$ be a finite Hopf algebra. Then every $Q\in\qquot(H)$ is
	$Q$-Galois.
\end{corollary}
Furthermore, the following proposition holds:
\begin{proposition}\label{prop:smontgomery}
	Let \(H\) be a Hopf algebra. Then there is a bijective correspondence:
	\[ \bigg\{K\subseteq H:\,K\,-\text{right coideal subalgebra}\bigg\}\galois{\simeq}{}\bigg\{H/I:\,I\,-\text{left ideal coideal}\bigg\} \]
	if and only if 
	\begin{enumerate}[topsep=0pt,noitemsep] 
	    \item for every its generalised quotient $Q$ the extension
		$^{co\,Q}H\subseteq H$ is $Q$-Galois
	    \item $\;^{co\,H/K^+H}\!H\subseteq K$ for every right coideal
		subalgebra $K$ of $H$.
	\end{enumerate}
\end{proposition}
\begin{proof}
    The pair of maps:
    \begin{center}
	\begin{tikzpicture}
	    \matrix[matrix of math nodes, column sep=2cm, row sep=.4cm]{
	    |(A)| \qsub(H)	&|(B)| \qquot(H)\\};
	    \node[anchor=east] at (-.8cm,-.9cm) {$K$};
	    \node[anchor=west] at (.8cm,-.9cm) {$H/HK^+$};
	    \node[anchor=east] at (-.8cm,-1.3cm) {$^{co\,Q}H$};
	    \node[anchor=west] at (.8cm,-1.3cm) {$Q$};
	    \begin{scope}
		\draw[->] ($(A)+(1.2cm,.1cm)$) -- node[above]{$\psi$} ($(B)+(-1.3cm,.1cm)$);
		\draw[<-] ($(A)+(1.2cm,-.1cm)$) -- node[below]{$\phi$} ($(B)+(-1.3cm,-.1cm)$);
		\draw[|->] (-.8cm,-.9cm) -- (.8cm,-.9cm);
		\fill[color=white] (0,-.9cm) circle (.2cm);
		\node at (0,-.9cm) {$\psi$};
		\draw[<-|] (-.8cm,-1.3cm) -- (.8cm,-1.3cm);
		\fill[color=white] (0,-1.3cm) circle (.2cm);
		\node at (0,-1.3cm) {$\phi$};
	    \end{scope}
	\end{tikzpicture}
    \end{center}
    is a Galois connection. The Proposition~\ref{prop:mono} and the above
    result shows that $\phi$ is a monomorphism and thus by the Galois property
    $\psi\phi=\id$. In the presence of a Galois correspondence, the equality
    $\phi\psi=\id$ is equivalent to the inclusion $^{co\,H/K^+H}H\subseteq K$.
\end{proof}

\section{Cleft extensions}
Let us introduce cleft extensions:
\begin{definition}
    An \(H\)-extension \(A/B\) is called \textsf{cleft} if there exists
    a convolution invertible \(H\)-comodule map \(\gamma:H\sir A\). 
\end{definition}
Here we give definition of the normal basis property for $H$-extensions.
\begin{definition}
	Let $B\subseteq A$ be an $H$-extension. Then it has the \textsf{normal
	    basis property} if and only if $A$ is isomorphic to $B\otimes_kH$ as left
	$B$-module and right $H$-comodule.
\end{definition}
\citet[Example 8.2.2]{sm:hopf-alg} shows that an extension of
fields has the classical normal basis property if and only if it has the above
property. There is a characterisation of cleft extensions due to Y.~Doi and M.~Takeuchi
(\cite[see also][]{rb-sm:crossed-products}):
\begin{theorem}[\cite{yd-mt:cleft-comodule-algebras}]
	Let $B\subseteq A$ be an $H$-extension. Then it is \textsf{cleft} if
	and only if it is a Hopf--Galois extension with normal basis property.
\end{theorem}
\noindent Theorem~\ref{thm:newTakeuchi} together with
Proposition~\ref{prop:closed-of-qquot} yields the following result. 
\begin{theorem}\label{thm:cleft-case}
    Let \(A/B\) be an \(H\)-cleft extension. Then an element $Q$ of
    $\qquot(H)$ is \textsf{closed} in the Galois connection~\eqref{eq:main} if
    and only if the extension $A/A^{co Q}$ is \textsf{$Q$-Galois}. When $H$ is
    finite dimensional then there is a \textsf{bijective correspondence}: 
    \[\Sub_{\textit{H\text{-}ext}}(A)\simeq\qquot(H)\]
\end{theorem}
\noindent The normality condition can be added to both sides as it is done in
Theorem~\ref{thm:newTakeuchi}. The first part is a consequence of Takeuchi
characterisation of cleft extensions, the formula \(\psi(B\otimes K)=H/K^+H\)
for \(B\otimes K\subseteq A=B\otimes H\), where \(K\) is right ideal
subalgebra of \(H\) and Proposition~\ref{prop:closed-of-qquot}. The bijection
is a direct consequence of properties of Galois connections and
Theorem~\ref{thm:newTakeuchi}(3). 
\section{BiGalois extensions}
In this section we show that (left, right) admissible quotients of \(H\) and
\(L(A,H)\) classifies the same subextensions of \(A\).

In~\cite{ps:hopf-bigalois} there is constructed a Hopf algebra $L(A,H)$ for
a given $k\subseteq A$ $H$-Hopf--Galois extension with the property that
$k\subseteq A$ is $(L(A,H),H)$-biGalois extension. This construction extends
the one given by~\cite{fo-yz:gal-cor-hopf-galois} to the non-commutative case.
The Hopf algebra $L(A,H)$ is unique up to isomorphism. Its underlying algebra
is $(A\otimes A)^{co\,H}$ (under the codiagonal coaction of $H$~on~$A\otimes
A$),~\cite[see][Theorem~3.5]{ps:hopf-bigalois} for more details. We only
include here relevant parts of the whole Galois theory based on this
additional Hopf algebra. For closer acknowledgement we refer the reader to the
papers
of~\citet{ps:hopf-bigalois,ps:gal-cor-hopf-bigal,ps:hopf-galois-and-bigalois}
and also the work of~\citet{fo-yz:gal-cor-hopf-galois}. 
\begin{definition}
	We call $A/k$ a \textsf{$L$-$H$-biGalois extension} if $A$ is left $L$
	comodule algebra, $A/k$ is a left $L$-Hopf--Galois extension and $A$
	is a right $H$-comodule algebra such that $A/k$ is a right
	$H$-Hopf--Galois extension. Moreover, $A$ is supposed to be
	$L$-$H$-bicomodule so that both coactions commute.
\end{definition}
Let us define notions which plays an important role in a Galois connection
between $\qquot(L)$ and $\Sub_{\textit{alg}^H}(A)$ -- the complete lattice of
$H$-subcomodule algebras of $A$.

\begin{definition}\label{defi:admissible}
	Fix a coalgebra $C$ and its quotient coalgebra $C/I$ where $I$ is an
	coideal of $C$. The quotient coalgebra $C/I$ is \textsf{right (left)
	    admissible} if it is flat over $k$ (thus faithfully flat) and $C$
	is right (left) faithfully coflat over $C/I$. We call a coideal $I$ of
	$C$ \textsf{right (left) admissible} if $C/I$ is. A bialgebra or Hopf
	algebra quotient is admissible if it is admissible as a coalgebra.

	A subalgebra $B$ of $A$ is \textsf{right (left) admissible} if $A$ is
	faithfully flat over $B$ as right (left) module. In both cases
	\textsf{admissible} will mean left and right admissible.
\end{definition}
We refer to~\cite{ps:gal-cor-hopf-bigal} for more on admissibility of
subalgebras and quotients. 
\begin{proposition}[{\citet[Proposition 3.2 and Theorem 3.6]{ps:gal-cor-hopf-bigal}}]
	Let $A/k$ be a faithfully flat $L(H,A)$-$H$-biGalois extension of a
	ring $k$. Then there exists a \textsf{Galois connection}:
$$\qquot(L)\;\galois{\mathcal F}{\mathcal I}\;\Sub_{\textit{alg}^H}(A)$$ 
	such that $\mathcal F(L/I)=\,^{co\,L/I}A$ and $\mathcal
	I(B)=(A\otimes_B A)^{co\,H}$. If in addition the antipodes of $H$ and
	$L(H,A)$ are bijective then \textsf{admissible objects} are
	\textsf{closed elements} of the Galois connection $(\mc F,\mc I)$. The
	bijection between closed objects restricts to the admissible objects.
\end{proposition}
\noindent It is shown in~\cite[Corollary~3.6]{ps:hopf-bigalois} that the antipode of
$L(A,H)$ is bijective if the antipode of $H$ is bijective and $A/k$ is
faithfully flat.

Let us denote by $A^{op}$ the opposite algebra to an algebra $A$ (the
underlying vector space of $A^{op}$ is the same as $A$ but the multiplication
is precomposed with the flip of tensor factors subsequently denoted by $\tau$).
An opposite bialgebra $B^{op}$ has opposite multiplication and comultiplication
(i.e. $\Delta^{op}:=\tau\circ\Delta,\mu^{op}=\mu\circ\tau$). The opposite
bialgebra of a Hopf algebra is a Hopf algebra if and only if the antipode is
bijective.  Then $S^{op}=S^{-1}$.

\begin{theorem}
	Let $H$ be a Hopf algebra over a field $k$ with bijective antipode.
	Let $A/A^{co\,H}$ be a faithfully flat $H$-Hopf--Galois extension.
	Then the map of the Galois connection~\eqref{eq:Main-Theorem}
	\mbox{$\phi:\qquot(H)\mpr{}\Sub_{\textit{alg}}(A/A^{co\,H})$} $Q\mapsto
	A^{co\,Q}$ is \textsf{injective} on the set of (right, left) admissible
	quotients of $H$ and its image is in the set of (right, left)
	admissible subalgebras of $A$. 
\end{theorem}

\begin{proof}
If $Q$ is left admissible ($H$ is faithfully coflat as left $Q$-comodule) then
by~\cite[Remark~1.4~(2)]{hs:normal-bases} $A$ is a $Q$-Galois extension.
Moreover, $A$ is faithfully flat as a left $A^{co\,Q}$-module
by~\cite[Theorem~1.4(2)]{hs:normal-bases} and thus $\phi(Q)=A^{co\,Q}$ is left
admissible. Now if $Q$ is right admissible then $Q^{op}$ is left admissible
for $H^{op}$ ($A^{op}$ is a left $H^{op}$-Galois extension) and by the same
reasoning as in~\cite[Theorem 1.4 (2)]{hs:normal-bases} we get that $A^{op}$
is a left faithfully flat $Q^{op}$-Galois extension. Then $A$ is a right
faithfully flat $Q$-Galois extension. Therefore, $A^{co\,Q}$ is right
admissible. By Proposition~\ref{prop:mono} the map $\phi$ is injective on the
set of (right, left) admissible quotients of $H$.
\end{proof}

\begin{corollary}\label{cor:equal}
	Let $k\subseteq A$ be a $L(H,A)$-$H$-biGalois extension of a field $k$,
	where $H$ is a Hopf algebra with a bijective antipode. The last two
	results show:
\begin{center}
\begin{tikzpicture}
\node[anchor=east] (A) at (0cm,0cm) {\begin{tabular}{ll}\textsf{(left, right) admissible}\\\textsf{quotients of $L(H,A)$}\end{tabular}}; 
\node[anchor=west] (B) at (3cm,0cm) {\begin{tabular}{ll}\textsf{(left, right) admissible }\\\textsf{$H$-comodule subalgebras of $A$}\end{tabular}};
\node[anchor=east] (C) at (0cm,-2cm) {\begin{tabular}{ll}\textsf{(left, right) admissible}\\\textsf{quotients of $H$}\end{tabular}};
\node[anchor=west] (D) at (3cm,-2cm) {\begin{tabular}{ll}\textsf{(left, right) admissible}\\\textsf{subalgebras of $A$}\end{tabular}};
\begin{scope}
	\draw[>->,>=angle 60,thick] 	(A) -- node[above]{$\simeq$} (B);
	\draw[>->,>=angle 60,thick] 	(C) -- (D);
\end{scope}
\node at (5cm,-1cm) {\rotatebox{-90}{$\subseteq$}};
\end{tikzpicture}
\end{center}
\end{corollary}

%

\section*{Acknowledgements}

We would like to thank T.~Brzezi\'{n}ski for reading our manuscript and giving
his comments and advices. 

\bibliographystyle{plainnat}
\bibliography{Mat}
\end{document}